\documentclass[11pt]{amsart}

\usepackage{amssymb}
\usepackage{mathrsfs}
\usepackage{amsfonts}
\usepackage{latexsym,amsmath,amsthm,amssymb,amsfonts}
\usepackage[usenames]{color}
\usepackage{xspace,colortbl}
\usepackage{graphicx}

\newcommand{\be}{\begin{equation}}
\newcommand{\ee}{\end{equation}}
\newcommand{\beq}{\begin{eqnarray}}
\newcommand{\eeq}{\end{eqnarray}}

\newtheorem{prop}{Proposition}[section]

\newtheorem{remark}[prop]{Remark}

\def\begeq{\begin{equation}}
\def\endeq{\end{equation}}

\def\odot{\setbox0=\hbox{$\bigcirc$}\relax \mathbin {\hbox
to0pt{\raise.5pt\hbox to\wd0{\hfil $\wedge$\hfil}\hss}\box0 }}

\numberwithin{equation} {section}

\numberwithin{equation}{section}
\newtheorem{theorem}{\bf Theorem}[section]
\newtheorem{proposition}[theorem]{\bf Proposition}
\newtheorem{definition}[theorem]{\bf Definition}
\newtheorem{lemma}[theorem]{\bf Lemma}

\newtheorem{corollary}[theorem]{\bf Corollary}

\allowdisplaybreaks

\begin{document}

\title[inverse curvature flows] {non-parametric inverse curvature flows in the AdS-Schwarzschild manifold}

\author{
Li Chen
 and Jing Mao}
\address{
Faculty of Mathematics and Statistics, Hubei University, Wuhan
430062,  P.R. China. } \email{chernli@163.com, jiner120@163.com}

\thanks{
This research was supported in part by the National Natural Science
Foundation of China (11201131,11401131) and Hubei Key Laboratory of
Applied Mathematics (Hubei University).}
%\author{Guofang Wang}
%\address{ Albert-Ludwigs-Universit\"at Freiburg,
%Mathematisches Institut Eckerstr. 1 D-79104 Freiburg}
%\email{guofang.wang@math.uni-freiburg.de}

%\

\date{}
%\maketitle
\begin{abstract} We consider the inverse curvature
flows in  the anti-de Sitter-Schwarzschild manifold with star-shaped
initial hypersurface, driven by the 1-homogeneous curvature
function. We show that the solutions exist for all time and the
principle curvatures of the hypersurface converges to 1
exponentially fast.
\end{abstract}

\maketitle {\it \small{{\bf Keywords}: Inverse curvature flows,
AdS-Schwarzschild manifold, homogeneous curvature function.}

}

\section{Introduction}
During the past decades, geometric flows have been studied
intensively. Following the ground breaking work of Huisken
\cite{Hu}, who considered the mean curvature flow, several authors
started to investigate inverse, or expanding curvature flows of
star-shaped closed hypersurfaces in ambient spaces of constant or
asymptotically constant sectional curvature. Gerhardt \cite{Ge1} and
Urbas \cite{Ur} independently considered flows of the form
\begin{equation}\label{0}
\frac{d}{dt}X=\frac{1}{F}\nu
\end{equation}
in $\mathbb{R}^{n+1}$, where $F$ is a curvature function homogeneous
of degree 1, and proved that the flow exists for all time and
converges to infinity. After a proper rescaling, the rescaled flow
will converge to a sphere.

The equation \eqref{0} has the property that it is scale-invariant
which seems to be the underlying reason why expanding curvature
flows in Euclidean space do not develop singularities contrary to
contracting curvature flows which will contract to a point in finite
time (see \cite{Hu}). Similar convergence results for inverse
curvature flows in the hyperbolic space were estimated by Ding
\cite{Di} and Gerhardt \cite{Ge2}, and in the sphere  by Gerhardt
\cite{Ge5} and Makowski-Scheuer \cite{Ma}. In \cite{Di}, Ding  also
get similar results in rotationally symmetric spaces of Euclidean
volume growth except the hyperbolic space. Compared with
scale-invariant flows, there may be some difference for
non-scale-invariant inverse curvature flows (see \cite{Ur1},
\cite{Ge6} and \cite{Sc1}).

It is a natural question, whether one can prove long-time existence
and the flow hypersurfaces become umbilic as in case of more general
ambient spaces. Recently, Brendle-Hung-Wang \cite{Be} investigated
the inverse mean curvature flow (IMCF for short) in anti-de
Sitter-Schwarzschild manifold which is asymptotically hyperbolic at
the infinity, and applied the convergence result to prove a sharp
Minkowski inequality for strictly mean convex and star-shaped
hypersurface in anti-de Sitter-Schwarzschild manifold. Similar
applications can be found in the works \cite{Gw1} and \cite{Li}, in
which the IMCF was used to prove a Minkowski type inequality in the
anti-de Sitter-Schwarzschild manifold and in the Schwarzschild
manifold respectively. Other geometric inequalities,  e.g.,
Aleksandrov-Fenchel inequalities in hyperbolic space as in
\cite{Gw2,Gw3} have been proven also using inverse 1-homogeneous
curvature flows \cite{Ge2} (also compare with \cite{Ma}, where
additional isoperimetric type problems have been treated).

In the present work, we investigate the convergence of the flow
\eqref{0} in some asymptotically hyperbolic space. More precisely,
we consider the convergence of the flow \eqref{0} in anti-de
Sitter-Schwarzschild manifold which is asymptotically hyperbolic at
the infinity. Recently, Lu \cite{Lu} considered the inverse hessian
quotient curvature flow with star-shaped initial hypersurface in the
anti-de Sitter-Schwarzschild manifold and proved that the solution
exists for all time, and the second fundamental form converges to
identity exponentially fast.

Let us first recall the definition of the anti-de
Sitter-Schwarzschild manifold (see also \cite{Be}). Fixed a real
number $m>0$, and let $s_{0}$ denote the unique positive solution of
the equation $1+s_{0}-ms_{0}^{1-n}=0$. The anti-de
Sitter-Schwarzschild manifold is an $(n+1)$-dimensional manifold
$M=[s_{0}, +\infty)\times\mathbb{S}^{n}$ equipped with the
Riemannian metric
$$\overline{g}=\frac{1}{1-ms^{1-n}+s^2}ds\otimes ds+s^{2}g_{\mathbb{S}^{n}},$$
where $g_{\mathbb{S}^{n}}$ is the standard round metric on the unit
sphere $\mathbb{S}^{n}$. Clearly, $\overline{g}$ is asymptotically
hyperbolic, since the sectional curvatures of $(M, \overline{g})$
approach -1 near infinity.

The anti-de Sitter-Schwarzschild manifold are examples of the static
spaces. If we define $f=\sqrt{1-ms^{1-n}+s^{2}}$, then it satisfies
the equation
\begin{equation}\label{1.1}
(\overline{\Delta}f)\overline{g}-\overline{\nabla}^{2}f+fRic=0.
\end{equation}
 In general, a Riemannian metric is called static if it satisfies
 \eqref{1.1} for some positive function $f$. The condition
 \eqref{1.1} guarantees the Lorentzian warped product $-f^2 dt\otimes
 dt+\overline{g}$ is a solution of the Einstein equation.

In order to formulate the main result, we need a definition below
(see also \cite{Sc1}).
\begin{definition}\label{1.2}
Let $\Gamma \subset \mathbb{R}^{n}$ be an open, symmetric and convex
cone and $F\in C^{\infty}(\Gamma)$  be a symmetric function. A
hypersurface $\Sigma_0$ in the anti-de Sitter-Schwarzschild manifold
$(M, \overline{g})$ is called F-admissable, if at any point $x\in
\Sigma_0$ the principal curvatures of $\Sigma_0$, $\kappa_1$, ...,
$\kappa_n$, are contained in the cone $\Gamma$.
\end{definition}

 We mainly get the following result

\begin{theorem}\label{main}
Let $\Gamma \subset \mathbb{R}^{n}$ be an open, symmetric and convex
cone that satisfies
$$\Gamma_{+}=\{(\kappa_i)\in \mathbb{R}^{n}: \kappa_{i}>0, \ \ \forall 1\leq i \leq n\}\subset \Gamma$$
and $F\in C^{\infty}(\Gamma)\cap C^{0}(\overline{\Gamma})$ be a
monotone, 1-homogeneous and concave curvature function, such that
$$F|_{\Gamma}>0 \quad \mbox{and} \quad F|_{\partial\Gamma}=0.$$
We usually normalized $F$ such that $$F(1,...,1)=n.$$ Let $\Sigma_0
$ be a smooth, star-shaped and $F$-admissable embedded closed
hypersurface in AdS-Schwarzschild manifold $(M, \overline{g})$, and
$\Sigma_0$ can be written as a graph over a geodesic sphere
$\mathbb{S}^{n}$,
$$\Sigma_0=\emph{graph} \ r(0, .).$$
Then

(1) There is a unique smooth curvature flow
$$X: [0, \infty)\times \Sigma\rightarrow M,$$
which satisfies the flow equation
\begin{equation}\label{3.01}
\left\{\begin{array}{l}
\frac{d}{dt}X=\frac{1}{F}\nu,\\
X(0)=\Sigma_0.
\end{array}\right.
\end{equation}
where $\nu(t,\xi)$ is the outward normal to $\Sigma_{t}=X(t,
\Sigma)$ at $X(t,x)$, $F$ is evaluated at the principle curvatures
of $\Sigma_t$ at $X(t,\xi)$ and the leaves $\Sigma_t$ are graphs
over $\mathbb{S}^{n}$,
$$\Sigma_{t}=\emph{graph }\ r(t,.).$$

(2) The leaves $\Sigma_{t}$ become more and more umbilic, namely
$$\mid h^{i}_{j}-\delta^{i}_{j}\mid\leq C e^{-\frac{2t}{n}}.$$

(3) Furthermore, the function $$\widetilde{r}(t, \theta)=\ r(t,
\theta)-\frac{t}{n} $$ converges to a well defined function
$f(\theta)\in C^{2}(\mathbb{S}^{n})$ in $C^{2, \alpha}$ as
$t\rightarrow+\infty$, which implies that the limit of the rescaled
induced metric of $\Sigma_t$ is the conformal metric
$e^{2f}g_{\mathbb{S}^n}$ on $\mathbb{S}^n$, where $g_{\mathbb{S}^n}$
is the round metric $\mathbb{S}^n$.
\end{theorem}

\begin{remark}
Similar to \cite{Hung} and \cite{Ne}, in general, the function
$f(\theta)$ in Theorem \ref{main} may not be constant in the sense
that the limit shape of the rescaled flow hypersurfaces does not
have to be a round sphere.
\end{remark}

The main techniques employed here were from \cite{Ge2} and later
were developed by Scheuer in \cite{Sc1}.

\textbf{Acknowledgement:} The authors would like to express
gratitude to Professor Guofang Wang for some suggestive comments and
they also thank Dr. Hengyu Zhou for pointing out the lost curvature
terms in the Codazzi equation.

\section{Graphic hypersurfaces in the AdS-Schwarzschild manifold and a reformulation of the problem}

First, we state some general facts about the AdS-Schwarzschild
manifold and the graphic hypersurfaces in it. We basically follow
the description in \cite[Section 2]{Be}. Denote the
AdS-Schwarzschild manifold by  $(M, \overline{g})$ and
$\overline{\nabla}$ by the Levi-Civata connection with respect to
the metric $\overline{g}$. By a change of variable, the
AdS-Schwarzschild metric can be rewritten as
$$\overline{g}=dr\otimes dr+\lambda(r)^{2}g_{\mathbb{S}^{n}},$$
where $\lambda(r)$ satisfies the ODE
\begin{equation}\label{2.001}
\lambda^{\prime}(r)=\sqrt{1+\lambda^2-m\lambda^{1-n}}
\end{equation}
and the asymptotic expansion
\begin{equation}\label{2.01}
\lambda(r)=sinh(r)+\frac{m}{2(n+1)}sinh^{-n}(r)+O(sinh^{-n-2}(r)).
\end{equation}
We can calculate the asymptotic expansion of Riemannian curvature
tensors. Let $e_{\alpha}$, $\alpha=1,2,...,n+1$, be an orthonormal
frame and $\overline{R}_{\alpha\beta\gamma\mu}$ denote the
Riemannian curvature tensor of the AdS-Schwarzschild metric. Then
\begin{equation}\label{2.1}
\overline{R}_{\alpha\beta\gamma\mu}=-\delta_{\beta\mu}\delta_{\alpha\gamma}
+\delta_{\beta\gamma}\beta_{\alpha\mu}+O(e^{-(n+1)r})
\end{equation}
and
\begin{equation}\label{2.2}
\overline{\nabla}_{\rho}\overline{R}_{\alpha\beta\gamma\mu}=O(e^{-(n+1)r}).
\end{equation}

Since $\Sigma\subset M$ is a graphic hypersurface in $M$, it can be
parametrized by
$$\Sigma=\{(r(\theta), \theta): \theta\in \mathbb{S}^{n}\}$$
for some smooth function $r$ on $\mathbb{S}^{n}$. Let
$\theta=\{\theta^{i}\}_{i=1,...,n}$ be a local coordinate system on
$\mathbb{S}^{n}$ and let $\partial_{i}$ be the corresponding
coordinate vector fields on $\mathbb{S}^{n}$ and
$\sigma_{ij}=g_{\mathbb{S}^{n}}(\partial_{i},
\partial_{j})$.Let $\varphi_{i}=D_{i}\varphi$, $\varphi_{ij}=D_{j}D_{i}\varphi$ and
$\varphi_{ijk}=D_{k}D_{j}D_{i}\varphi$ denote the covariant
derivatives of $\varphi$ with respect to the round metric
$g_{\mathbb{S}^{n}}$ and $\nabla$ be the Levi-Civata connection of
$\Sigma$ with respect to the induced metric $g$ from $(M,
\overline{g})$. Set $X=(r(\theta), \theta)$, the tangential vectors
on $\Sigma$ take the form
$$X_i=\partial_i+D_{i}r \partial_r.$$
The induced metric on $\Sigma$ is
$$g_{ij}=r_i r_j+\lambda^2\sigma_{ij},$$
and the outward unit normal vector of $\Sigma$
$$\nu=\frac{1}{v}\bigg(\partial_{r}-\lambda^{-2}D^{j}r\partial_{j}\bigg).$$
Define a new function $\varphi: \mathbb{S}^{n}\rightarrow\mathbb{R}$
by
\begin{equation}\label{02.01}
\varphi(\theta)=\int_{c}^{r(\theta)}\frac{1}{\lambda(s)}ds.
\end{equation}
Then the induced metric on $\Sigma$ takes the form
$$g_{ij}=\lambda^2(\varphi_i\varphi_j+\sigma_{ij})$$
with the inverse
$$g^{ij}=\lambda^{-2}(\sigma^{ij}-\frac{\varphi^i\varphi^j}{v^2}),$$
where $(\sigma^{ij})=(\sigma_{ij})^{-1}$,
$\varphi^i=\sigma^{ij}\varphi_j$ and
\begin{equation}\label{2.7}
v^2=1+\sigma^{ij}\varphi_i\varphi_j\equiv1+\mid
D\varphi\mid^2=1+\frac{\mid  D r\mid^2}{\lambda^2},
\end{equation}
$|.|$ is the norm corresponding to the metric $g_{\mathbb{S}^n}$.
Let $h_{ij}$ be the second fundamental form of $\Sigma\subset M$ in
term of the coordinate $\theta^i$. So
$$h_{ij}=\frac{\lambda}{v}\bigg(\lambda^{\prime}(\varphi_i\varphi_j+\sigma_{ij})-\varphi_{ij}\bigg)$$
and
\begin{equation}\label{2.07}
h^{i}_{j}=\frac{1}{\lambda
v}(\lambda^{\prime}\delta^{i}_{j}-\widetilde{g}^{ik}\varphi_{kj}),
\end{equation}
where
$\widetilde{g}^{ij}=\sigma^{ij}-\frac{\varphi^i\varphi^j}{v^2}$.

To calculate some curvature terms, we need the following result from
Appendix A in \cite{Pe}.
\begin{lemma}
$$\overline{R}(\partial_i, \partial_j, \partial_k, \partial_l)=\lambda^2(1-(\lambda^{\prime})^2)(\sigma_{ik}\sigma_{jl}-\sigma_{il}\sigma_{jk})$$
$$\overline{R}(\partial_i, \partial_r, \partial_j, \partial_r)=-\lambda\lambda^{\prime\prime}\sigma_{ij}$$
\end{lemma}
Then, we can calculate some curvature terms by using the above
lemma.
\begin{eqnarray}\label{c1}
\overline{R}(X_i, \nu, X_j, \nu)&\equiv&\overline{R}_{i\nu
j\nu}=-\frac{1}{v^2}\big[\lambda
\lambda^{\prime\prime}+((\lambda^{\prime})^2-1)\mid D
\varphi\mid^2\big]\sigma_{ij}\\
\nonumber&&-\frac{1}{v^2}\big[\frac{2\lambda^{\prime\prime}}{\lambda}+\frac{1-(\lambda^{\prime})^2}
{\lambda^2}+\frac{\lambda^{\prime\prime}}{\lambda} \mid D
\varphi\mid^2\big]r_i r_j
\end{eqnarray}
and
\begin{eqnarray}\label{c2}
\overline{R}(\nu, X_i, X_k, X_j)\equiv\overline{R}_{\nu i k
j}=(-\lambda\lambda^{\prime\prime}-(1-(\lambda^{\prime})^2))\frac{r_k
\sigma_{ij}}{v}
+(\lambda\lambda^{\prime\prime}+(1-(\lambda^{\prime})^2))\frac{r_j
\sigma_{ik}}{v}.
\end{eqnarray}
Thus,
\begin{eqnarray}\label{c3}
v \overline{R}_{\nu i k j}r^k=\frac{1}{v^2
\lambda^2}(\lambda\lambda^{\prime\prime}+(1-(\lambda^{\prime})^2))[r_i
r_j-\lambda^2 \mid D \varphi\mid^2 \sigma_{ij}]
\end{eqnarray}
The geodesic spheres $S_{r}$ in the AdS-Schwarzschild manifold $(M,
\overline{g})$ are totally umbilic, their second fundamental form is
given by
$$\overline{h}_{ij}=\overline{h}_{ij}(r)=\frac{\lambda^{\prime}}{\lambda}\overline{g}_{ij},$$
where
$$\overline{g}_{ij}=\lambda^{2}\sigma_{ij}.$$
Thus
$\overline{h}^{i}_{j}=\frac{\lambda^{\prime}}{\lambda}\delta^{i}_{j}$,
$\kappa_{i}=\frac{\lambda^{\prime}}{\lambda}$ and the mean curvature
$\overline{H}$ of $S_r$ is given by
$$\overline{H}=\overline{H}(r)=\frac{n \lambda^{\prime}}{\lambda}.$$

For the evolution of graphic hypersurfaces, we can reform the
equation \eqref{0}. Let $\Sigma_0$ be a graphic hypersurface in
AdS-Schwarzschild manifold $(M, \overline{g})$ which is given by an
embedding
$$X_0: \mathbb{S}^{n}\rightarrow M.$$
Let $X_t: \mathbb{S}^{n}\rightarrow M$, $t\in[0, T)$, be the
solution of inverse curvature flow with initial data given by $X_0$.
In other word,
\begin{equation}\label{3.1}
\frac{\partial X}{\partial t}=\frac{1}{F}\nu,
\end{equation}
where $\nu$ is the outward unit normal vector and $F$ is a monotone,
1-homogeneous and concave curvature function. We shall call
\eqref{3.1} the parametric form of the flow. We can write the
initial hypersurface $\Sigma_0$ as the graph of a function $r_0$
defined on the unit sphere:
$$\Sigma_0=\{(r_0(\theta), \theta): \theta\in \mathbb{S}^{n}\}.$$
If each $\Sigma_t$ is graphic, it can be parametrized as follows
$$\Sigma_t=\{(r(\theta, t), \theta): \theta\in \mathbb{S}^{n}\}.$$
Then the evolution equation \eqref{3.1} now yields
$$\frac{dr}{dt}=\frac{1}{F v} \quad \mbox{and}  \quad  \frac{d \theta^i}{dt}=-\frac{D^{i}r}{\lambda^{2}Fv},$$
from which we deduce
\begin{equation}\label{3.2}
\frac{\partial r}{\partial t}=\frac{v}{F},
\end{equation}
where $v$ is given by \eqref{2.7}. Therefore, as long as the
solution of \eqref{3.01} exists and remain graphic, it is equivalent
to a parabolic PDE \eqref{3.2} for $r$. The equation \eqref{3.2} is
also referred as the non-parametric form of the inverse mean
curvature flow. Notice that the velocity vector of \eqref{3.01} is
always normal, while the velocity vector of \eqref{3.2} is in the
direction of $\partial_r$. To go from one to the other, we take the
difference which is a (time-dependent) tangential vector field and
compose the flow of the reparametrization associated with the
tangent vector field.

The proof of the short time existence of the flow \eqref{3.01} is
standard, see Remark 3.4 in \cite{Sc1} and Remark 2.1 in \cite{Sc2}.
For completeness, we describe it easily here.We can get the short
time existence of the flow on a maximal interval $[0, T^{*})$,
$0<T^{*}\leq\infty$, and
$$X\in C^{\infty}([0, T^{*})\times \Sigma, M).$$
Moreover, all the leaves $M(t)=X(t, M)$, $0\leq t<T^{*}$, are
admissable and can be written as graphs over $\mathbb{S}^{n}$.
Furthermore, the flow $X$ exists as long as the scalar flow
\eqref{3.2} does, where
$$r: [0, T^*)\times \mathbb{S}^{n}\longrightarrow \mathbb{R}.$$
Thus, we will mainly investigate the long time existence of
\eqref{3.01} in the following chapters.

\section{the long-time existence }

The proof of the long-time existence of \eqref{3.01} is standard
which mainly relies on the following $C^0$ estimates, $C^1$
estimates and curvature estimates. Before proceeding, we give some
notation. Covariant differentiation will usually be denoted by
indices, e.g. $r_{ij}$ for a function $r: \Sigma\rightarrow
\mathbb{R}$, or, if ambiguities are possible, by a semicolon, e.g.
$h_{ij;k}$. Usual partial derivatives will be denoted by a comma,
e.g. $u_{i,j}$.

\textbf{$C^0$ estimates}

\

First, we recall the $C^0$ estimates whose proof is standard, see
Lemma 3.1 in \cite{Ge2} and Section 4 in \cite{Di}.
\begin{lemma}\label{4.01}
The solution $r$ of \eqref{3.2} satisfies
\begin{equation}\label{4.1}
\lambda(\inf r(0,\cdot))\leq\lambda(r(t, \theta
))e^{-\frac{t}{n}}\leq\lambda(\sup r(0,\cdot)), \ \ \ \forall
\theta\in \mathbb{S}^{n}, t\in [0, T^*).
\end{equation}

\end{lemma}

\begin{remark}
Noticing the asymptotic expansion \eqref{2.01} of $\lambda(r)$, we
have from the above lemma
\begin{equation}\label{re4.1}
r(t, \theta)-\frac{t}{n}=o(t).
\end{equation}
\end{remark}

\

\textbf{$C^1$ estimates}

\

To get the $C^1$ estimate, we using the the evolution equation of
$\varphi$ instead of $r$ by noticing the relation \eqref{2.01}. From
\eqref{3.2}, we get
\begin{equation}\label{4.2}
\frac{\partial \varphi}{\partial t}=\frac{v}{\lambda
F(h^{i}_{j})}=\frac{v}{F(\lambda h^{i}_{j})}\equiv
\frac{v}{F(\widetilde{h}^{i}_{j})}.
\end{equation}
Let $(\widetilde{g}_{ij})=(\widetilde{g}^{ij})^{-1}$, clearly,
$g_{ij}=\lambda^2\widetilde{g}_{ij}$. Defining
$$\widetilde{h}_{ij}=\widetilde{g}_{ik}\widetilde{h}^{k}_{j},$$
we see that in \eqref{4.2} we are considering the eigenvalues of
$\widetilde{h_{ij}}$ with respect to $\widetilde{g}_{ij}$ and thus
we define
$$F^{ij}=\frac{\partial F}{\partial \widetilde{h}_{ij}} \ \ \ \mbox{and}\ \ \
F^{i}_{j}=\frac{\partial F}{\partial \widetilde{h}^{j}_{i}}.$$ By a
straightforward computation, it is easy to get the following
relations.
\begin{lemma}\label{le4.2}
\begin{equation*}\label{4.3}
\widetilde{h}^{l}_{k;i}=-\frac{v_i}{v}\widetilde{h}^{l}_{k}-v^{-1}(\widetilde{g}^{lm}_{\
\ ;i}\varphi_{mk}+\widetilde{g}^{lm}\varphi_{mki}-\lambda
\lambda^{\prime\prime}D_{i}\varphi \delta^{l}_{k}),
\end{equation*}
\begin{equation*}\label{4.4}
\widetilde{g}^{kl}_{;\
i}=\frac{2v_{i}\varphi^{k}\varphi^{l}}{v^3}-\frac{1}{v^2}\bigg(\varphi^{k}_{i}\varphi^{l}+\varphi^{k}\varphi^{l}_{i}\bigg),
\end{equation*}
\begin{equation*}\label{4.5}
v_{i}=v^{-1}\varphi_{ki}\varphi^k,
\end{equation*}
where the covariant derivatives as well as index raising are
performed with respect to $\sigma_{ij}$.
\end{lemma}

\begin{lemma}\label{4.5}
Let $\varphi$ be a solution of \eqref{4.2}, we have
\begin{equation}\label{4.6}
|D\varphi|^2 \leq \sup_{\mathbb{S}^n}|D\varphi(0,\cdot)|^2.
\end{equation}
Moreover, if $F$ is bounded from above $F\leq C$, then there exists
$0 <\mu=\mu(C)$ such that
\begin{equation}\label{4.61}
|D\varphi|^2 \leq e^{-\mu
t}\sup_{\mathbb{S}^n}|D\varphi(0,\cdot)|^2.
\end{equation}
\end{lemma}

\begin{proof}
Let $$w=\frac{1}{2}|D\varphi|^2.$$ By differentiating \eqref{4.2}
with respect to the operator $D^{k}\varphi D_k$, we obtain
$$\frac{\partial}{\partial t}w=-\frac{v}{F^2}F^{k}_{l}\widetilde{h}^{l}_{k; i}\varphi^i+\frac{v_i\varphi^i}{F}.$$
Fix $0<T<T^*$ and suppose
$$\sup_{[0, T]\times\mathbb{S}^{n}}w=w(t_0, \xi_0), \ \  t_0>0.$$
Then at $(t_0, \xi_0)$, there holds
\begin{eqnarray*}
0\leq \frac{\partial}{\partial
t}w&=&-\frac{1}{F^2}\bigg(-\widetilde{g}^{lm}_{\ \ ;i}
\varphi_{mki}\varphi^i-\widetilde{g}^{lm}\varphi_{mki}+\lambda
\lambda^{\prime\prime}|D\varphi|^2
\delta^{l}_{k}\bigg)+\frac{2}{v^3}\varphi_{ki}\varphi^{k}\varphi^{i}\\
&=&-\frac{2}{F^2}\lambda
\lambda^{\prime\prime}F^{kl}\widetilde{g}_{kl}w+\frac{1}{F^2}F^{kl}\varphi_{kli}\varphi^i,
\end{eqnarray*}
where we use Lemma \ref{le4.2} and the fact
$\varphi_{ik}\varphi^i=0, \ \forall k$ at $(t_0, \xi_0)$. Then, we
apply the rule for exchanging derivatives
$$\varphi_{kli}=\varphi_{ikl}+R_{ilkm}\varphi^m$$
and notice the fact on $\mathbb{S}^n$
$$R_{ilkm}=\sigma_{ik}\sigma_{lm}-\sigma_{im}\sigma_{lk},$$
we can obtain
\begin{eqnarray*}
0\leq \frac{\partial}{\partial t}w=\frac{1}{F^2}\bigg(-2\lambda
\lambda^{\prime\prime}F^{kl}\widetilde{g}_{kl}w+F^{kl}(\varphi_{k}\varphi_l
-|D\varphi|^2\sigma_{kl})+F^{kl}w_{kl}-F^{kl}\varphi_{ik}\varphi^{i}_{l}\bigg)<0,
\end{eqnarray*}
where we use the assumption that $F$ is a monotone, 1-homogeneous
and concave curvature function and $F^{kl}w_{kl}\leq0$ at $(t_0,
\xi_0)$. Hence, the estimate \eqref{4.6} follows by the
arbitrariness of $T$. To prove \eqref{4.61}, we define
\begin{eqnarray*}
\widetilde{w}=we^{-\mu t},
\end{eqnarray*}
where $\mu$ is a positive constant which will be chosen later. Then
$\widetilde{w}$ satisfies the same equation as $w$ with an
additional term $\mu \widetilde{w}$ at the right-hand side. Assume
$\widetilde{w}$ attains a positive maximum at a point $(t_0, \xi_0),
\ t_0>0$, by applying the maximum principle as before, there holds
\begin{eqnarray}\label{201}
0\leq -\frac{2}{F^2}\lambda
\lambda^{\prime\prime}F^{kl}\widetilde{g}_{kl}\widetilde{w}+\mu\widetilde{w}.
\end{eqnarray}
Then, since
$\frac{\lambda^{\prime\prime}}{\lambda}=1+\frac{1}{2}m(n-1)\lambda^{1-n}$
is bounded by some constant $C_1$ from Lemma \ref{4.01},
$F(\widetilde{h}^{i}_{j})\lambda^{-1}=F(h^{i}_{j})$ is bounded from
above and $F^{kl}\widetilde{g}_{kl}\geq n$, we can obtain
\begin{eqnarray*}
we^{\mu t}\leq \sup_{\mathbb{S}^n}w(0)
\end{eqnarray*}
for all
\begin{eqnarray*}
0<\mu\leq \frac{C_1 n}{C^2}.
\end{eqnarray*}
\end{proof}

\begin{remark}
In Theorem \ref{800} below, we will estimate the optimal decay rate
$\mu$.
\end{remark}

\

\textbf{Curvature estimates}

\

In this section, for convenience, we let $\Phi=\Phi(F)=-\frac{1}{F}$
, $\Phi^{\prime}=\frac{d \Phi}{d F}$ and
$$\chi=\langle\lambda\frac{\partial}{\partial r}, \nu\rangle=\frac{\lambda}{v}$$
\begin{lemma}
Under the flow \eqref{3.01}, the following evolution equations hold
true
\begin{equation}
\frac{\partial}{\partial
t}\Phi-\Phi^{\prime}F^{ij}\Phi_{ij}=\Phi^{\prime}F^{ij}h_{ik}h^{k}_{j}\Phi+
\Phi^{\prime}F^{ij}\overline{R}_{\nu i \nu j}\Phi,
\end{equation}
\begin{equation}
\frac{\partial}{\partial t}
r-\Phi^{\prime}F^{ij}r_{ij}=2\Phi^{\prime}Fv^{-1}-
\Phi^{\prime}F^{ij}\overline{h}_{ij},
\end{equation}
\begin{equation}\label{4.2001}
\frac{\partial}{\partial t}
\chi-\Phi^{\prime}F^{ij}\chi_{ij}=\Phi^{\prime}F^{ij}h_{i}^{k}h_{kj}-\Phi^{\prime}F^{ij}
\overline{R}(\nu, X_i, (\lambda\partial_r)^{T}, X_j) ,
\end{equation}
\begin{equation}\label{4.200}
\frac{\partial}{\partial t}h^{i}_{j}=\Phi^{i}_{j}+\Phi
h^{i}_{k}h^{k}_{j}+\Phi \overline{R}_{\nu j \nu k}g^{ki},
\end{equation}
where $\partial_r=\frac{\partial X}{\partial r}$,
$X_i=\frac{\partial X}{\partial \xi^{i}}$ and
$(\lambda\partial_r)^{T}=\lambda\partial_r-\langle\lambda\partial_r,
\nu\rangle\nu$.
\end{lemma}

\begin{proof}
This is a straightforward computation in any case by using the flow
equation \eqref{3.01}. For details, we can see the similar results
in \cite{Ge2} for the flow in hyperbolic space.
\end{proof}

\begin{proposition}\label{F1}
Let $X$ be a solution of the inverse curvature flow \eqref{3.01}.
Then the curvature function is bounded from above, i.e. there exists
$C=C(n, \Sigma_0)$ such that
\begin{equation}
F(t, \xi)\leq C(n, \Sigma_0)<\infty \ \ \forall (t, \xi)\in[0,
T^*)\times \Sigma.
\end{equation}
\end{proposition}

\begin{proof}
The proof proceeds similarly to that in Lemma 4.2 in \cite{Ge2}. Let
$$w=-\log(-\Phi)+\beta(r-\frac{t}{n}),$$
where $\beta$ is supposed to be large. Fix $0<T<T^*$ and suppose
$$\sup_{[0, T]\times\mathbb{S}^{n-1}}w=w(t_0, \xi_0), \ \  t_0>0.$$
Then at $(t_0, \xi_0)$, there holds
\begin{eqnarray*}
0=w_i=-\frac{\Phi_i}{\Phi}+cr_i
\end{eqnarray*}
and
\begin{eqnarray*}
0\leq \frac{\partial}{\partial
t}w-\Phi^{\prime}F^{ij}w_{ij}&=&-\Phi^{\prime}F^{ij}h_{ik}h^{k}_{j}
-\Phi^{\prime}F^{ij}\overline{R}_{\nu i \nu
j}-\Phi^{\prime}F^{ij}\frac{\Phi_i
\Phi_j}{\Phi^2}\\
&&+2\beta\Phi^{\prime}F
v^{-1}-\beta\Phi^{\prime}F^{ij}\overline{h}_{ij}-\frac{1}{n}.
\end{eqnarray*}
Thus, we have
\begin{eqnarray*}
0\leq \Phi^{\prime}F^{ij}\bigg(-\overline{R}_{\nu i \nu j}-\beta^2
r_ir_j-\beta\frac{\lambda^{\prime}}{\lambda}\lambda^2\sigma_{ij}\bigg)+\beta\bigg(\frac{2}{F
v}-\frac{1}{n}\bigg).
\end{eqnarray*}
It is easy to see from \eqref{2.001}, \eqref{2.01} and Lemma
\ref{4.01}
\begin{eqnarray}\label{00.1}
\frac{\lambda^{\prime}}{\lambda}=1+O(e^{-\frac{n+1}{n}t}),
\end{eqnarray}
\begin{eqnarray}\label{00.2}
\frac{\lambda^{\prime\prime}}{\lambda}=1-\frac{1}{2}m(1-n)\lambda^{-n-2}=1+O(e^{-\frac{n+2}{n}t})
\end{eqnarray}
and
\begin{eqnarray}\label{00.3}
\frac{1-(\lambda^{\prime})^2}{\lambda^2}=-1+m\lambda^{-n-1}=-1+O(e^{-\frac{n+1}{n}t}).
\end{eqnarray}
Combing the above three estimates, as $\beta$ is supposed to be
large, we can get from \eqref{c1}
\begin{eqnarray*}
\Phi^{\prime}F^{ij}\bigg(-\overline{R}_{\nu i \nu j}-\beta^2
r_ir_j-\beta\frac{\lambda^{\prime}}{\lambda}\lambda^2\sigma_{ij}\bigg)\leq
0.
\end{eqnarray*}
Therefore, we can obtain
\begin{eqnarray*}
0\leq \beta\bigg(\frac{2}{F v}-\frac{1}{n}\bigg).
\end{eqnarray*}
Then,
\begin{eqnarray*}
F(t_0, \xi_0)\leq C(n, \Sigma_0),
\end{eqnarray*}
which leads to
\begin{eqnarray*}
w\leq C(n, \Sigma_0).
\end{eqnarray*}
Therefore, the inequality
\begin{eqnarray*}
F\leq C(n, \Sigma_0)
\end{eqnarray*}
holds.
\end{proof}

\begin{proposition}\label{F2}
Let $X$ be a solution of the inverse curvature flow \eqref{3.01}.
Then the curvature function is bounded from below, i.e., there
exists $C=C(n, \Sigma_0)$ such that
\begin{equation}
0<C(n, \Sigma_0)\leq F(t, \xi), \qquad \ \ \forall (t, \xi)\in[0,
T^*)\times \Sigma.
\end{equation}
\end{proposition}

\begin{proof}
The proof proceeds similarly to that of \cite[Lemma 4.1]{Ge2}. Let
$$w=\log(-\Phi)-\log(\chi e^{-\frac{t}{n}}).$$
Fix $0<T<T^*$ and suppose
$$\sup_{[0, T]\times\mathbb{S}^{n}}w=w(t_0, \xi_0), \ \  t_0>0.$$
Then at $(t_0, \xi_0)$, there holds
\begin{eqnarray*}
0=w_i=\frac{\Phi_i}{\Phi}-\frac{\chi_i}\chi{},
\end{eqnarray*}
which leads to
\begin{eqnarray*}
0\leq \frac{\partial}{\partial
t}w-\Phi^{\prime}F^{ij}w_{ij}=\Phi^{\prime}\chi^{-1}F^{ij}\overline{R}(\nu,
X_i, \lambda\partial_r, X_j)+\frac{1}{n}.
\end{eqnarray*}
Then, we can have by using \eqref{c1} and \eqref{c3}
\begin{eqnarray}\label{c4}
\chi^{-1}F^{ij}\overline{R}(\nu, X_i, \lambda\partial_r,
X_j)&=&F^{ij}\overline{R}(\nu, X_i, \nu,
X_j)+vF^{ij}\overline{R}(\nu, X_i, X_k, X_j)r_{l}g^{kl}\\
\nonumber&=&-\frac{\lambda^{\prime\prime}}{\lambda}F^{ij}g_{ij}.
\end{eqnarray}
Therefore,
\begin{eqnarray*}
0\leq \frac{\partial}{\partial
t}w-\Phi^{\prime}F^{ij}w_{ij}=-\frac{\lambda^{\prime\prime}}{\lambda}
\Phi^{\prime}F^{ij}g_{ij}+\frac{1}{n},
\end{eqnarray*}
Since $F^{ij}g_{ij}\geq F(1, , ,1)=n$, we have from the estimate
\eqref{00.2}
\begin{eqnarray*}
0<C(n, \Sigma_0)\leq F(t_0, \xi_0).
\end{eqnarray*}
Thus,
\begin{eqnarray*}
w\leq w(t_0, \xi_0)\leq C(n, \Sigma_0).
\end{eqnarray*}
From \eqref{4.1}, we know there exists $C(n, \Sigma_0)>0$ such that
\begin{eqnarray*}
C^{-1}\leq\chi e^{-\frac{t}{n}}\leq C.
\end{eqnarray*}
Therefore, the inequality
\begin{eqnarray*}
0<C(n, \Sigma_0)\leq F
\end{eqnarray*}
holds.
\end{proof}

Now we begin to estimate the second fundamental form which is the
most difficult part of the proof of the long-time existence. The
proof is similar to that of \cite[Lemma 4.4]{Ge2}, but due to the
non-vanishing term $\overline{\nabla}_{i}\overline{R}_{jklm}$ in
non-constant curvature manifolds, our case is more complicated and
needs a far more delicate treatment.

\begin{proposition}\label{4.07}
Let $X$ be a solution of the inverse curvature flow \eqref{3.01}.
Then, the principal curvatures of the flow hypersurfaces are
uniformly bounded from above, i.e., there exists $C=C(n, \Sigma_0)$
such that
\begin{equation*}
\kappa_i(t, \xi)\leq C(n, \Sigma_0), \qquad \ \ \forall (t,
\xi)\in[0, T^*)\times \Sigma.
\end{equation*}
\end{proposition}

\begin{proof}
First, we need the evolution equation of $h_{j}^{i}$. From
\eqref{4.200} we can get
\begin{eqnarray}\label{4.7}
\frac{\partial}{\partial
t}h_{j}^{i}=\Phi^{\prime}F^{kl}\nabla^{i}\nabla_{j}h_{kl}+\Phi^{\prime\prime}F^{i}F_{j}+F^{kl,
pq}h_{kl;}^{\ \ i}h_{pq;j}+\Phi
h^{i}_{k}h^{k}_{j}+\Phi\overline{R}_{\nu j \nu k}g^{ki}.
\end{eqnarray}
Using Gauss equation and Codazzi equation, we have
\begin{eqnarray}\label{111}
F^{kl}\nabla_{k}\nabla_{l}h_{ij}&=&F^{kl}\nabla_{i}\nabla_{j}h_{kl}+F^{kl}(\overline{R}_{kilp}h^{p}_{j}
+\overline{R}_{kjlp}h^{p}_{i})+2F^{kl}\overline{R}_{kijp}h^{p}_{l}\\
\nonumber &&-F^{kl}\overline{R}_{\nu j i
\nu}h_{kl}-F^{kl}\overline{R}_{\nu k \nu
l}h_{ij}+F^{kl}(\overline{\nabla}_{k}\overline{R}_{\nu i j
l}+\overline{\nabla}_{i}\overline{R}_{\nu ljk})\\ \nonumber
&&+F^{kl}h_{kl}h^{p}_{i}h_{pj}-F^{kl}h_{il}h_{k}^{p}h_{pj}+F^{kl}h_{kj}h_{i}^{p}h_{pl}-F^{kl}h_{kp}h^{p}_{l}h_{ij}.
\end{eqnarray}
Then, we get the evolution equation of $h_{j}^{i}$ by combing
\eqref{4.7} and \eqref{111}
\begin{eqnarray}\label{1111}
&&\frac{\partial}{\partial
t}h_{j}^{i}-F^{kl}\nabla_{k}\nabla_{l}h^{i}_{j}\\&=&-\Phi^{\prime}\bigg(F^{kl}(\overline{R}_{kqlp}h^{p}_{j}g^{qi}
+\overline{R}_{kjlp}h^{p}_{q}g^{qi})+2F^{kl}\overline{R}_{kqjp}h^{p}_{l}g^{qi}\nonumber \\
\nonumber &&-F^{kl}\overline{R}_{\nu j p
\nu}h_{kl}g^{pi}-F^{kl}\overline{R}_{\nu k \nu
l}h^{i}_{j}+F^{kl}(\overline{\nabla}_{k}\overline{R}_{\nu p j
l}g^{pi}+g^{pi}\overline{\nabla}_{p}\overline{R}_{\nu ljk})\\
\nonumber
&&+F^{kl}h_{kl}h^{pi}h_{pj}-F^{kl}h^{i}_{l}h_{k}^{p}h_{pj}+F^{kl}h_{kj}h^{ip}h_{pl}-F^{kl}h_{kp}h^{p}_{l}h^{i}_{j}\bigg)\nonumber
\\ \nonumber&&+\Phi^{\prime\prime}F^{i}F_{j}+F^{kl,
pq}h_{kl;}^{\ \ i}h_{pq;j}+\Phi
h^{i}_{k}h^{k}_{j}+\Phi\overline{R}_{\nu j \nu k}g^{ki}.
\end{eqnarray}
Using the estimates\eqref{4.1} and \eqref{4.6}, there exists a
constant $\vartheta>0$ such that
\begin{eqnarray*}
2\vartheta\leq \widetilde{\chi}\equiv\chi e^{-\frac{t}{n}}.
\end{eqnarray*}
Setting
\begin{eqnarray*}
\rho=-\log(\widetilde{\chi}-\vartheta).
\end{eqnarray*}
 By using the equation \eqref{4.2001}, we get the evolution of
 $\rho$ as follows
\begin{eqnarray*}
\frac{\partial}{\partial
t}\rho-\Phi^{\prime}F^{kl}\rho_{kl}=(\widetilde{\chi}-\vartheta)^{-1}\bigg(-\Phi^{\prime}
F^{kl}h_{k}^{p}h_{pl}\widetilde{\chi}+\frac{\widetilde{\chi}}{n}+\widetilde{\chi}\Phi^{\prime}F^{ij}\overline{R}(\nu,
X_i, (\lambda \partial_r)^{T}, X_j)\bigg)-
\Phi^{\prime}F^{kl}\frac{\widetilde{\chi}_{k}\widetilde{\chi}_{l}}{(\widetilde{\chi}-\vartheta)^2}.
\end{eqnarray*}

Next, we define the functions
\begin{eqnarray*}
\zeta=\sup\{h_{ij}\eta_i\eta_j: g_{ij}\eta^i \eta^j=1\}
\end{eqnarray*}
and
\begin{eqnarray*}
w=\log \zeta+\rho+\beta(r-\frac{t}{n}),
\end{eqnarray*}
where $\beta>0$ is supposed to be large. We claim that $w$ is
bounded, if $\beta$ is chosen sufficiently large. Fix $0<T<T^*$,
suppose $w$ attains a maximal value at $(t_0, \xi_0)$
$$\sup_{[0, T]\times\mathbb{S}^{n}}w=w(t_0, \xi_0), \ \  t_0>0.$$
Choose Riemannian normal coordinates at $(t_0, \xi_0)$ such that at
this point we have
\begin{eqnarray*}
g_{ij}=\delta_{ij}, \ \ \ h_{ij}=\kappa_i \delta_{ij}, \ \ \
\kappa_1\leq \kappa_2\leq...\leq \kappa_{n},
\end{eqnarray*}
then
\begin{eqnarray}\label{4.700}
F^{kl,pq}\eta_{kl}\eta_{pq}\leq \sum_{k\neq
l}\frac{F^{kk}-F^{ll}}{\kappa_k-\kappa_l}(\eta_{kl})^{2}\leq
\frac{2}{\kappa_n-\kappa_1}\sum_{i=1}^{n}(F^{nn}-F^{ii})(\eta_{ni})^{2}
\end{eqnarray}
and
\begin{eqnarray}\label{4.7000}
F^{nn}\leq \cdot\cdot\cdot\leq F^{11}.
\end{eqnarray}
For details, see, e.g., \cite[Lemma 1.1]{Ge4} and \cite[Lemma 2
]{Ec}.

Since $\zeta$ is only continuous in general, we need to find a
differential version instead. Set
\begin{eqnarray*}
\widetilde{\zeta}=\frac{h_{ij}\eta^i \eta^j}{g_{ij}\eta^i \eta^j},
\end{eqnarray*}
where $\eta=(0,...,0,1)$. There holds at $(t_0, \xi_0)$,
\begin{eqnarray*}
h_{n n}=h^{n}_{n}=\kappa_n=\zeta=\widetilde{\zeta}
\end{eqnarray*}
By a simple calculation, we find
\begin{eqnarray*}
\frac{\partial}{\partial
t}\widetilde{\zeta}=\frac{(\frac{\partial}{\partial t}h_{ij})\eta^i
\eta^j}{g_{ij}\eta^i \eta^j}-\frac{h_{ij}\eta^i
\eta^j}{(g_{ij}\eta^i \eta^j)^2}(\frac{\partial}{\partial
t}g_{ij})\eta^i \eta^j
\end{eqnarray*}
and
\begin{eqnarray*}
\frac{\partial}{\partial t}h^{n}_{n}=\frac{\partial}{\partial
t}(h_{nk}g^{kn})=(\frac{\partial}{\partial
t}h_{nk})g^{kn}-g^{ki}(\frac{\partial}{\partial
t}g_{ij})g^{jn}h_{nk}.
\end{eqnarray*}
Clearly, there holds in a neighborhood of $(t_0, \xi_0)$
\begin{eqnarray*}
\widetilde{\zeta}\leq\zeta
\end{eqnarray*}
and we find at $(t_0, \xi_0)$
\begin{eqnarray*}
\frac{\partial}{\partial
t}\widetilde{\zeta}=\frac{\partial}{\partial t}h^{n}_{n}
\end{eqnarray*}
and the spatial derivatives do also coincide. This implies that
$\widetilde{\zeta}$ satisfies the same evolution \eqref{4.7} as
$h^{n}_{n}$. Without loss of generality, we treat $h^{n}_{n}$ like a
scalar and pretend that $w$ is defined by
\begin{eqnarray*}
w=\log h^{n}_{n}+\rho+\beta(r-\frac{t}{n}).
\end{eqnarray*}
Using the asymptotic expansion of Riemannian curvature tensors
\eqref{2.2}, the non-vanishing terms
$\overline{\nabla}_{i}\overline{R}_{jklm}$ which appear in
\eqref{1111} can be fortunately controlled by
\begin{eqnarray*}
\mid F^{kl}(\overline{\nabla}_{k}\overline{R}_{\nu pj
l}g^{pi}+g^{pi}\overline{\nabla}_{p}\overline{R}_{\nu l j
k})\mid\leq C F^{pq}g_{pq}.
\end{eqnarray*}
Then, we get the evolution equation of $h_{n}^{n}$ from \eqref{1111}
\begin{eqnarray}\label{4.70}
&&\frac{\partial}{\partial t}\log
h_{n}^{n}-\Phi^{\prime}F^{kl}\nabla_{k}\nabla_{l}\log h^{n}_{n}\\
\nonumber&=&\frac{1}{\kappa_n}\bigg(\frac{\partial}{\partial t}
h_{n}^{n}-\Phi^{\prime}F^{kl}\nabla_{k}\nabla_{l}
h^{n}_{n}\bigg)+\Phi^{\prime}\frac{1}{\kappa_{n}^2}F^{kl}h_{n;
k}^{n}h_{n; l}^{n}
\\ \nonumber&\leq&\frac{1}{\kappa_n}\Phi^{\prime}\bigg(F^{kl}h_{kp}h^{p}_{l}\kappa_n-2F\kappa_{n}^2
-2F^{kl}\overline{R}_{knln}\kappa_{n}-2F^{kl}\overline{R}_{knnp}h^{p}_{l}\\
\nonumber&& +F\overline{R}_{\nu n n \nu}+F^{kl}\overline{R}_{\nu k
\nu l}\kappa_{n}+C F^{kl}g_{kl}-F\overline{R}_{\nu n \nu
n}\bigg)+\Phi^{\prime}\frac{1}{\kappa_{n}^2}F^{kl}h_{n; k}^{n}h_{n;
l}^{n}\\ \nonumber&&+F^{kl, pq}h_{kl;}^{\ \
i}h_{pq;j}+\Phi^{\prime\prime}F^{i}F_{j}.
\end{eqnarray}
Together with the evolution equations of $\rho$ and $r$ , we infer
at $(t_0, \xi_0)$, the following inequality
\begin{eqnarray}\label{4.79}
0&\leq&\Phi^{\prime}F^{kl}h_{kp}h^{p}_{l}(1-\frac{\widetilde{\chi}}{\widetilde{\chi}-\vartheta})
-2\Phi^{\prime}Fh^{n}_{n}+2\beta\Phi^{\prime}Fv^{-1}-\beta\Phi^{\prime}F^{ij}\overline{h}_{ij}-\frac{\beta}{n}
+\frac{1}{n}\frac{\widetilde{\chi}}{\widetilde{\chi}-\vartheta}\nonumber\\
&&+\Phi^{\prime}F^{kl}(\log h^{n}_{n})_k(\log
h^{n}_{n})_l-\Phi^{\prime}F^{kl}\rho_k
\rho_l+\frac{2}{\kappa_n-\kappa_1}\Phi^{\prime}\sum_{i=1}^{n}(F^{nn}-F^{ii})(h_{ni;}^{\
\ \  n})^2 (h_{n}^{n})^{-1}\\ \nonumber
&&+\frac{1}{\kappa_n}\Phi^{\prime}\bigg(
-2F^{kl}\overline{R}_{knln}\kappa_{n}-2F^{kl}\overline{R}_{knnp}h^{p}_{l}
+F^{kl}\overline{R}_{\nu k \nu l}\kappa_{n}+C
F^{kl}g_{kl}-2F\overline{R}_{\nu n \nu n}\bigg)\\ \nonumber&&
+\frac{\widetilde{\chi}}{\widetilde{\chi}-\vartheta}F^{ij}\overline{R}(\nu,
X_i, (\lambda \partial_r)^{T}, X_j)+\Phi^{\prime\prime}F^{i}F_{j}
\end{eqnarray}
holds. We can estimate the curvature terms by using \eqref{2.1}
$$\mid
-2F^{kl}\overline{R}_{knln}\kappa_{n}-2F^{kl}\overline{R}_{knnp}h^{p}_{l}
-F^{kl}\overline{R}_{\nu k \nu l}\kappa_{n}+C F^{kl}g_{kl}\mid\leq
C(1+\kappa_n)F^{kl}g_{kl}$$ and
$$\mid F\overline{R}_{\nu n \nu n}\mid \leq C F.$$
Then, using the inequalities \eqref{4.700} and \eqref{4.7000},
$\Phi^{\prime\prime}<0$ and
\begin{eqnarray*}
(\log h^{n}_{n})_i=-\rho_i-\beta r_i
\end{eqnarray*}
at $(t_0, \xi_0)$, we can get from the above inequality
\begin{eqnarray}\label{4.800}
0&\leq&\Phi^{\prime}F^{kl}h_{kp}h^{p}_{l}\frac{\vartheta}{\widetilde{\chi}-\vartheta}
+\Phi^{\prime}F^{kl}\bigg(C
g_{kl}(1+\kappa_{n}^{-1})-\beta\overline{h}_{kl}\bigg)-2\Phi^{\prime}Fh^{n}_{n}
\\&&+2\beta\Phi^{\prime}Fv^{-1}-\frac{\beta}{n}
+\bigg(\frac{1}{n}+F^{ij}\overline{R}(\nu,
X_i, (\lambda \partial_r)^{T}, X_j)\bigg)\frac{\widetilde{\chi}}{\widetilde{\chi}-\vartheta}\nonumber\\
&&+\beta^2\Phi^{\prime}F^{kl}r_k r_l-2\beta\Phi^{\prime}F^{kl}\rho_k
r_l+\frac{2}{\kappa_n-\kappa_1}\Phi^{\prime}\sum_{i=1}^{n}(F^{nn}-F^{ii})(h_{ni;}^{\
\ \  n})^2 (h_{n}^{n})^{-1}\nonumber\\ \nonumber
&&+C(\kappa_{n})^{-1}\Phi^{\prime}F.
\end{eqnarray}
Now, we estimate the left curvature term in the above inequality.
Clearly, we can get from \eqref{c3}
\begin{eqnarray}\label{a1}
\overline{R}(\nu, X_i, (\lambda \partial_r)^{T}, X_j)&=&\lambda
\overline{R}(\nu, X_i, X_k, X_j)r_l g^{kl}\\
\nonumber&=&\frac{1}{v^3
\lambda}(\lambda\lambda^{\prime\prime}+(1-(\lambda^{\prime})^2))[r_i
r_j-\lambda^2 \mid D \varphi\mid^2 \sigma_{ij}].
\end{eqnarray}
From \eqref{00.2}and \eqref{00.3}, we can get
\begin{eqnarray}\label{a2}
\frac{1}{\lambda}(\lambda\lambda^{\prime\prime}+(1-(\lambda^{\prime})^2))=(1+\frac{n-1}{2\lambda})\frac{m}{\lambda^n},
\end{eqnarray}
which is clearly bounded. Therefore,
$$F^{ij}\overline{R}(\nu, X_i, (\lambda \partial_r)^{T},
X_j)\leq CF^{ij}g_{ij}$$ Moreover, we know
\begin{eqnarray*}
F^{ij}\overline{h}_{ij}&=&\frac{\lambda^{\prime}}{\lambda}F^{ij}\overline{g}_{ij}=
\frac{\lambda^{\prime}}{\lambda}F^{ij}(g_{ij}-r_i
r_j)\geq\frac{\lambda^{\prime}}{\lambda}F^{ij}g_{ij}(1-g^{kl}r_k r_l)\\
&=&\frac{\lambda^{\prime}}{\lambda}v^{-2}F^{ij}g_{ij}\geq
C_0F^{ij}g_{ij},
\end{eqnarray*}
where we use \eqref{4.6} and \eqref{00.1} in the last inequality.
Furthermore, it is easy to check
\begin{eqnarray*}
v_i=v\frac{\overline{H}}{n}r_i-v^2h_{i}^{k}r_{k}
\end{eqnarray*}
(see (5.29) in \cite{Ge7}), and thus
\begin{eqnarray*}
\mid \nabla \rho \mid\leq C_2\mid \nabla v\mid +C_2\mid \nabla
r\mid\leq C_2|\kappa_n|\mid \nabla r\mid +C_2\mid \nabla r\mid,
\end{eqnarray*}
where $|\nabla \rho|=\sqrt{g^{ij}\nabla_{i}\rho\nabla_{j}\rho}$. We
distinguish two cases. \

\emph{Case 1}. If $\kappa_1<-\epsilon_1\kappa_n$,
$0<\varepsilon_1<1$, then
\begin{eqnarray*}
F^{kl}h_{kp}h^{p}_{l}\geq\frac{1}{n}F^{kl}g_{kp}\epsilon_{1}^{2}\kappa_{n}^{2}.
\end{eqnarray*}
Hence, after abandoning the negative term
$-2\Phi^{\prime}F\kappa_n$, \eqref{4.800} becomes
\begin{eqnarray*}\label{4.81}
0&\leq&\Phi^{\prime}F^{kl}g_{kl}\bigg(-\frac{1}{n}\epsilon_{1}^{2}\kappa_{n}^{2}\frac{\vartheta}
{\widetilde{\chi}-\vartheta}+C(1+\kappa_{n}^{-1})-\beta
C_0+C\frac{\widetilde{\chi}}{\widetilde{\chi}
-\vartheta}\nonumber\\
&&+2\beta C_{2}(\kappa_n+1)\mid \nabla r\mid^2+\beta^2\mid \nabla
r\mid^2\bigg)
\nonumber\\
&& 2\beta\Phi^{\prime}Fv^{-1}-\frac{\beta}{n}
+\frac{1}{n}\frac{\widetilde{\chi}}{\widetilde{\chi}-\vartheta}+C\kappa_{n}^{-1}\Phi^{\prime}F
.
\end{eqnarray*}
Sine $F$ bounded from above and below, $F^{ij}g_{ij}\geq
F(1,...,1)=n$ and $\mid \nabla r\mid=\frac{|D\varphi|}{v}\leq C(n,
\Sigma_0)$, the first line converges to $-\infty$ if
$\kappa_n\rightarrow +\infty$. Moreover, the last line is uniformly
bounded by some $C=C(n, \Sigma_0)$. Hence, in this case we conclude
that
\begin{eqnarray*}
\kappa_{n}\leq C(n, \Sigma_0)
\end{eqnarray*}
for any choice of $\beta$.

\emph{Case 2}. If $\kappa_1\geq-\epsilon_1\kappa_n$,
$0<\varepsilon_1<1$, then
\begin{eqnarray*}
\frac{2}{\kappa_n-\kappa_1}\Phi^{\prime}\sum_{i=1}^{n}(F^{nn}-F^{ii})(h_{ni;}^{\
\ \  n})^2 (h_{n}^{n})^{-1}\leq
\frac{2}{1+\epsilon_1}\Phi^{\prime}\sum_{i=1}^{n}(F^{nn}-F^{ii})(\log
h^{n}_{n})^{2}_{i}+C\Phi^{\prime}F^{ij}g_{ij}\kappa_{n}^{-2},
\end{eqnarray*}
where we use $h_{ni; n}=h_{nn; i}+\overline{R}_{n n i \nu}$ in view
of the Codazzi equation and the boundedness of the curvature
\eqref{1.2}. Thus, the terms in \eqref{4.79} containing the
derivatives of $h^{n}_{n}$ can therefore be estimated from above by
\begin{eqnarray*}
&&\Phi^{\prime}F^{ij}(\log h^{n}_{n})_{i}(\log
h^{n}_{n})_{j}+\frac{2}{\kappa_n-\kappa_1}\Phi^{\prime}\sum_{i=1}^{n}(F^{nn}-F^{ii})(h_{ni;}^{\
\ \  n})^2 (h_{n}^{n})^{-1}\\&&\leq
\frac{2}{1+\epsilon_1}\Phi^{\prime}\sum_{i=1}^{n}F^{nn}(\log
h^{n}_{n})^{2}_{i}-\frac{1-\epsilon_1}{1+\epsilon_1}\Phi^{\prime}\sum_{i=1}^{n}F^{ii}(\log
h^{n}_{n})^{2}_{i}+C\Phi^{\prime}F^{ij}g_{ij}\kappa_{n}^{-2}\\&&\leq
\frac{2}{1+\epsilon_1}\Phi^{\prime}\sum_{i=1}^{n}F^{nn}(\log
h^{n}_{n})^{2}_{i}-\frac{1-\epsilon_1}{1+\epsilon_1}\Phi^{\prime}\sum_{i=1}^{n}F^{nn}(\log
h^{n}_{n})^{2}_{i}+C\Phi^{\prime}F^{ij}g_{ij}\kappa_{n}^{-2}\\&&=\Phi^{\prime}F^{nn}|\nabla\rho+\beta
\nabla
r|^2+C\Phi^{\prime}F^{ij}g_{ij}\kappa_{n}^{-2}\\&&=\Phi^{\prime}F^{nn}(|\nabla\rho|^2+2\beta\langle\nabla\rho,
\nabla r\rangle+\beta^2|\nabla
r|^2)+C\Phi^{\prime}F^{ij}g_{ij}\kappa_{n}^{-2}.
\end{eqnarray*}
Hence, taking the above inequality into the estimate \eqref{4.79}
yields
\begin{eqnarray*}\label{4.779}
0&\leq&-\Phi^{\prime}F^{nn}\kappa_{n}^{2}\frac{\vartheta}{\widetilde{\chi}-\vartheta}
+\Phi^{\prime}F^{kl}g_{kl}\bigg(1-\beta
C_{0}+C(1+\kappa_{n}^{-1}+\kappa_{n}^{-2}+\frac{\widetilde{\chi}}{\widetilde{\chi}-\vartheta})
\bigg)\\&&-2\Phi^{\prime}F\kappa_{n}+2\beta\Phi^{\prime}Fv^{-1}-\frac{\beta}{n}
+\frac{1}{n}\frac{\widetilde{\chi}}{\widetilde{\chi}-\vartheta}+\Phi^{\prime}F^{nn}(2\beta|
\nabla \rho| |\nabla r|+\beta^2|
\nabla r|^2)\\
&&+C\kappa_{n}^{-1}\Phi^{\prime}F<0
\end{eqnarray*}
for large $\kappa_n$ if $\beta$ is chosen large enough. Thus we
obtain
\begin{eqnarray*}
\kappa_n(t_0, \xi_0)\leq C(n, \Sigma_0).
\end{eqnarray*}
Since $\rho$ and $\widetilde{r}$ are bounded from above, we conclude
our claim.
\end{proof}

\begin{corollary}\label{4.08}
Under the hypothesis of Proposition \ref{4.07}, there exists a
compact set $K\subset \mathbb{R}^n$ such that
\begin{equation*}
(\kappa_i)\subset K\subset\subset\Gamma.
\end{equation*}
\end{corollary}

\begin{proof}
Noticing that $F$ is bounded from below and $F^{ij}h_{ij}=F$,
Proposition \ref{4.07} implies the result.
\end{proof}

\begin{theorem}
Under the hypothesis of Theorem \ref{main}, we conclude
\begin{equation*}
T^{*}=+\infty.
\end{equation*}
\end{theorem}
\begin{proof}
Recalling that $\varphi$ satisfies the equation \eqref{4.2}
\begin{equation*}
\frac{\partial \varphi}{\partial t}=\frac{v}{\lambda
F(h^{i}_{j})}=G(x, \varphi, D\varphi, D^{2}\varphi).
\end{equation*}
By a simple calculation, we get
\begin{equation*}
\frac{\partial G}{\partial
\varphi^{i}_{j}}=\frac{1}{\lambda^{2}F^{2}}F^{j}_{k}\widetilde{g}^{k}_{i},
\end{equation*}
where $\widetilde{g}^{k}_{i}$ and $\delta^{k}_{i}$ are equivalent
norms, since $v\leq C$. Therefore, we can conclude the equation
\eqref{4.2} is uniformly parabolic on finite intervals from
Proposition \ref{F1}, Proposition \ref{F2} and Corollary \ref{4.08}.
Recalling that $h^{i}_{j}=\frac{1}{\lambda
v}(\lambda^{\prime}\delta^{i}_{j}-\widetilde{g}^{ik}\varphi_{kj}),$
where
$\widetilde{g}^{ij}=\sigma^{ij}-\frac{\varphi^i\varphi^j}{v^2}$, we
have
\begin{equation}\label{490}
|\varphi|_{C^2(\mathbb{S}^n)}\leq C(n, \Sigma_0, T^*)
\end{equation}
by using the estimate \eqref{4.6} and Corollary \ref{4.08}. Then by
Krylov-Safonov estimate \cite{Kr}, we have
\begin{equation*}
|\varphi|_{C^{2, \alpha}(\mathbb{S}^n)}\leq C(n, \Sigma_0, T^*),
\end{equation*}
which implies the maximal time interval is unbounded, i.e.,
$T^*=+\infty$.
\end{proof}

\

\textbf{Optimal decay estimates}

\

First, we recall \cite[Lemma 4.2]{Sc1} which will be used in the
next lemma.

\begin{lemma}\label{Ju}
Let $f\in C^{0, 1}(\mathbb{R}_{+})$ and let $D$ be the set of points
of differentiability of $f$. Suppose that for all $\epsilon>0$ there
exist $T_\epsilon>0$ and $\delta_\epsilon>0$ such that
\begin{eqnarray*}
A_\epsilon=\{t\in[T, +\infty)\cap D: f(t)\geq \epsilon\}\subset
\{t\in[T_\epsilon, +\infty)\cap D: f^{\prime}(t)\geq
-\delta_\epsilon\}.
\end{eqnarray*}
Then there holds
$$
\lim_{t\rightarrow\infty}\sup f(t)\leq 0.$$
\end{lemma}

\begin{lemma}\label{4.088}
Under the hypothesis of Theorem \ref{main}, the principle curvatures
of the flow hypersurfaces converges to $1$,
\begin{equation*}
\sup_{\Sigma}|\kappa_{i}(t, \cdot)-1|\rightarrow 0, \ \
t\rightarrow\infty, \ \ \forall 1\leq i\leq n.
\end{equation*}
\end{lemma}

\begin{proof}
We use the method which first appears in \cite{Sc1}. Define the
functions
\begin{eqnarray*}
\zeta=\sup\{h_{ij}\eta_i\eta_j: g_{ij}\eta^i \eta^j=1\}
\end{eqnarray*}
and
\begin{eqnarray*}
w=(\log \zeta-\log\widetilde{\chi}+\widetilde{r}-\log2)t,
\end{eqnarray*}
where $\widetilde{\chi}=\chi e^{-\frac{t}{n}}$ and
$\widetilde{r}=r-\frac{t}{n}$. We claim that $w$ is bounded. Fix
$0<T<+\infty$, suppose $w$ attains a maximal value at $(t_0,
\xi_0)$,
$$\sup_{[0, T]\times\mathbb{S}^{n-1}}w=w(t_0, \xi_0), \ \  t_0>0.$$
Choose Riemannian normal coordinates at $(t_0, \xi_0)$ such that at
this point we have
\begin{eqnarray*}
g_{ij}=\delta_{ij}, \ \ \ h_{ij}=\kappa_i \delta_{ij}, \ \ \
\kappa_1\leq \kappa_2\leq...\leq \kappa_{n}.
\end{eqnarray*}
Then it follows
\begin{eqnarray*}
w=(\log h_{n}^{n}-\log\widetilde{\chi}+\widetilde{r}-\log2)t.
\end{eqnarray*}

First, we claim that
\begin{eqnarray*}
(-\log\widetilde{\chi}+\widetilde{r}-\log2)t=(\log v-\log
\lambda+r-\log2)t=(\log v-\log 2\lambda+r)t
\end{eqnarray*}
is bounded. On the one hand, using the estimate \eqref{4.61},
\begin{eqnarray*}
t\log v=\log(1+v-1)^{t}\leq \log(1+Ce^{-\mu t})^{t}
\end{eqnarray*}
is bounded. On the other hand, the asymptotic expansions
\eqref{2.01} and \eqref{re4.1} imply
\begin{eqnarray*}
e^{(-\log
2\lambda+r)t}=(1-e^{-2r}+o(e^{-2r}))^{-t}\leq(1-Ce^{-\frac{2t}{n}})^{-t}
\end{eqnarray*}
is also bounded. Therefore, we prove our claim.

Using the evolution equations of $h^{n}_{n}$, $\widetilde{\chi}$ and
$\widetilde{r}$, as \eqref{4.79}, we can obtain the following
evolution equation of $w$

\begin{eqnarray}\label{4.79}
&&\frac{\partial}{\partial
t}w-\Phi^{\prime}F^{ij}w_{ij}\nonumber
\\&=&\bigg(-2\Phi^{\prime}Fh^{n}_{n}+2\Phi^{\prime}Fv^{-1}-\Phi^{\prime}F^{ij}\overline{h}_{ij}-\Phi^{\prime}\kappa_{n}F^{kl}(\log
h^{n}_{n})_k(\log h^{n}_{n})_l\nonumber \\
\nonumber &&-\Phi^{\prime}F^{kl}(\log \widetilde{\chi})_k (\log
\widetilde{\chi})_l+\Phi^{\prime}F^{kl, pq}h_{kl; n}h_{pq;}^{\ \
n}(h_{n}^{n})^{-1}\\ \nonumber
&&+\frac{1}{\kappa_n}\Phi^{\prime}\big(
-2F^{kl}\overline{R}_{knln}\kappa_{n}-2F^{kl}\overline{R}_{knnp}h^{p}_{l}
+F^{kl}\overline{R}_{\nu k \nu l}\kappa_{n}+
F^{kl}(\overline{\nabla}_{k}\overline{R}_{\nu n n
l}+\overline{\nabla}_{n}\overline{R}_{\nu lnk})-2F\overline{R}_{\nu
n \nu n}\big)\\ \nonumber&&+F^{ij}\overline{R}(\nu, X_i, (\lambda
\partial_r)^{T},
X_j)+\Phi^{\prime\prime}F^{i}F_{j}\bigg)t_0\nonumber\\
\nonumber && +\log
h_{n}^{n}-\log\widetilde{\chi}+\widetilde{r}-\log2.
\end{eqnarray}
Using the asymptotic expansion of Riemannian curvature tensors
\eqref{2.2} and \eqref{2.1}, we have
\begin{eqnarray*}
\mid F^{kl}(\overline{\nabla}_{k}\overline{R}_{\nu i m
l}g^{mj}+\overline{\nabla}^{n}\overline{R}_{\nu l n k})\mid\leq C
e^{-\frac{n+1}{n}t}.
\end{eqnarray*}
and
$$-2F^{kl}\overline{R}_{knln}\kappa_{n}-2F^{kl}\overline{R}_{knnp}h^{p}_{l}
-F^{kl}\overline{R}_{\nu k \nu l}\kappa_{n}-2F\overline{R}_{\nu n
\nu n}=F^{kl}g_{kl}\kappa_n+O(e^{-\frac{n+1}{n}t}).$$ Moreover, we
can get from \eqref{a1} and \eqref{a2}
$$|F^{ij}\overline{R}(\nu, X_i, (\lambda \partial_r)^{T},
X_j)|\leq Ce^{-t}.$$
Therefore, we have
\begin{eqnarray}\label{4.079}
\frac{\partial}{\partial t}w-\Phi^{\prime}F^{ij}w_{ij}&\leq&\bigg(
\Phi^{\prime}F^{kl}g_{kl}-2\Phi^{\prime}Fh^{n}_{n}+2\Phi^{\prime}Fv^{-1}-\Phi^{\prime}F^{ij}\overline{h}_{ij}
+\Phi^{\prime\prime}F_{n}F^{n}(h_{n}^{n})^{-1}
\nonumber\\
\nonumber &&+\Phi^{\prime}F^{kl}(\log h^{n}_{n})_k(\log
h^{n}_{n})_l-\Phi^{\prime}F^{kl}(\log\widetilde{\chi})_k (\log\widetilde{\chi})_l+\Phi^{\prime}F^{kl, pq}h_{kl; n}h_{pq;}^{\ \ n}(h_{n}^{n})^{-1}\bigg)t_{0}\nonumber\\
\nonumber &&+(\log
h_{n}^{n}-\log\widetilde{\chi}+\widetilde{r}-\log2)+O(1)
\nonumber\\
\nonumber &\leq& \Phi(2h_{n}^{n}-2v^{-1})t_0+\Phi^{\prime}F^{kl}(r_k
r_l+(1-\frac{\lambda^{\prime}}{\lambda})\lambda^{2}\sigma_{ij})t_0\nonumber\\
\nonumber &&+\Phi^{\prime}\bigg((\log h^{n}_{n})_k(\log
h^{n}_{n})_l-(\log \widetilde{\chi})_k(\log
\widetilde{\chi})_l\bigg)t_0+O(1)\nonumber\\
\nonumber && +\log
h_{n}^{n}-\log\widetilde{\chi}+\widetilde{r}-\log2.
\end{eqnarray}
Using inequalities \eqref{4.700} and \eqref{4.7000},
$\Phi^{\prime\prime}<0$ and
\begin{eqnarray*}
(\log h^{n}_{n})_i=-(\log \widetilde{\chi})_i-\widetilde{r}_i
\end{eqnarray*}
at $(t_0, \xi_0)$, we can get from the above inequality
\begin{eqnarray}\label{4.80}
0&\leq&\Phi(2h_{n}^{n}-2v^{-1})t_0+\Phi^{\prime}F^{kl}r_k
r_l t_0\\
\nonumber &&+\Phi^{\prime}F^{kl}(\log \widetilde{\chi})_k r_l
t_0+O(1)+\log h_{n}^{n}-\log\widetilde{\chi}+\widetilde{r}-\log2.
\end{eqnarray}
From \eqref{2.07}, we have
\begin{eqnarray*}
v_{k}=\frac{\varphi^{j}\varphi_{jk}}{v}=\lambda^{\prime}v\varphi_{k}
-\lambda v^{2}h^{i}_{k}\varphi_{i}=\frac{\lambda^{\prime}}{\lambda}v
r_{k} - v^{2}h^{i}_{k}r_{i}.
\end{eqnarray*}
Then, we obtain
\begin{eqnarray*}
(\log\widetilde{\chi})_{k}=\frac{\chi_k}{\chi}=\frac{v}{\lambda}\frac{\lambda^{\prime}r_{k}v-\lambda
v_k}{v^2}=vh^{i}_{k}r_{k}.
\end{eqnarray*}
Since the principal curvatures are bounded by Corollary \ref{4.08}
and $F$ is also bounded by Propositions \ref{F1} and \ref{F2}, the
following two terms in \eqref{4.80} are controlled by
\begin{eqnarray}\label{4.8000}
\Phi^{\prime}F^{kl}r_k r_l t_0+\Phi^{\prime}F^{kl}(\log
\widetilde{\chi})_k r_l t_0\leq C g^{ij}r_i r_j t_0.
\end{eqnarray}
However, $g^{ij}r_i r_j t_0=\frac{|D\varphi|^2}{1+|D\varphi|^2}\leq
Ce^{-\mu t_0}t_0\leq C(n, \Sigma_0)$ by Lemma \ref{4.5}. Therefore,
from \eqref{4.80} at $(t_0, \xi_0)$, we get
\begin{eqnarray*}\label{4.1180}
0\leq\Phi(2h_{n}^{n}-2v^{-1})t_0+C
\end{eqnarray*}
for some $C=C(n, \Sigma_0)$, which implies
\begin{eqnarray*}
h_{n}^{n}\leq 1+\frac{C F}{t_0}.
\end{eqnarray*}
Thus, we have
\begin{eqnarray*}
w\leq t_{0}\log(1+\frac{C
F}{t_0})+t_{0}(-\log\widetilde{\chi}+\widetilde{r}-\log2)\leq C(n,
\Sigma_0),
\end{eqnarray*}
which means $w$ has a priori boundness. Hence,
\begin{eqnarray}\label{481}
\limsup_{t\rightarrow\infty}\sup_{M}\kappa_{i}(t, \cdot)\leq 1,
\qquad \forall 1\leq i\leq n.
\end{eqnarray}

Now we define the function
\begin{eqnarray*}
\psi=\log(-\Phi)-\log\widetilde{\chi}+\widetilde{r}-\log
2-\log\frac{1}{n}.
\end{eqnarray*}
By a similar computation to that in the proofs of Propositions
\ref{F1} and \ref{F2}, we know that $\psi$ satisfies
\begin{eqnarray*}
\frac{\partial}{\partial t}\psi-\Phi^{\prime}F^{ij}\psi_{ij}
&=&\Phi^{\prime}F^{ij}(\log(-\Phi))_{i}(\log(-\Phi))_{j}-
\Phi^{\prime}F^{ij}(\log\widetilde{\chi})_{i}(\log\widetilde{\chi})_{j}\\
&&+\frac{1}{\chi}\overline{R}(\nu, \partial_i, \lambda\partial_{r},
\partial_j)+\frac{2}{F v}-\Phi^{\prime}F^{ij}\overline{h}_{ij}.
\end{eqnarray*}
Then the Lipschitz function
\begin{eqnarray*}
\widetilde{\psi}=\sup_{\xi\in\Sigma}\psi(\cdot, \xi)
\end{eqnarray*}
satisfies
\begin{eqnarray*}
\frac{\partial}{\partial t}\widetilde{\psi}&\leq& Ce^{-\mu
t}-\Phi^{\prime}F^{ij}g_{ij}(1+O(e^{-\frac{n+2}{n}t}))+\frac{2}{F
v}-\Phi^{\prime}F^{ij}\overline{h}_{ij}\\&\leq& Ce^{-min\{\mu,
\frac{n+2}{n}\}t}+\Phi^{\prime}(\frac{2F}{v}-2F^{ij}g_{ij}),
\end{eqnarray*}
where we use a similar argument which has been done to
\eqref{4.8000} to get the first inequality by noticing \eqref{c4}
and \eqref{00.2}. Setting
\begin{eqnarray*}
A_\epsilon=\{t\in[T, +\infty)\cap D: \widetilde{\psi}(t)\geq
\epsilon\},
\end{eqnarray*}
where $D$ is the set of points of differentiability of
$\widetilde{\psi}$. Let $\epsilon>0$ and choose $T>0$ such that for
all $(t, \xi)\in[T, \infty)\times \Sigma$,
\begin{eqnarray*}
-\log\widetilde{\chi}+\widetilde{r}-\log 2<\frac{\epsilon}{2}.
\end{eqnarray*}
Then we have
\begin{eqnarray*}
\bigg(\log(-\Phi)-\log\frac{1}{n}\bigg)(t,
\xi_{t})>\frac{\epsilon}{2}
\end{eqnarray*}
for $t\in A_\epsilon$, where $\widetilde{\psi}(t)=\psi(t, \xi_{t})$.
Thus there exists
$0<\gamma=\gamma(\epsilon)=n(1-e^{-\frac{\epsilon}{2}})$ such that
\begin{eqnarray*}
F(t, \xi_{t})<n-\gamma,
\end{eqnarray*}
which implies
\begin{eqnarray*}
\Phi^{\prime}(\frac{2F}{v}-2F^{ij}g_{ij})\leq-
\Phi^{\prime}\frac{2n\gamma}{v}.
\end{eqnarray*}
Therefore, if $T$ is chosen large enough, we have
\begin{eqnarray*}
\frac{\partial}{\partial t}\widetilde{\psi}\leq-
\frac{1}{2}(\inf\Phi^{\prime})\frac{2n\gamma}{v}\equiv
\delta_\epsilon.
\end{eqnarray*}
Now it follows from Lemma \ref{Ju},
\begin{eqnarray*}
\lim_{t\rightarrow\infty}\sup\widetilde{\psi}(t)\leq 0.
\end{eqnarray*}
Hence, we have
\begin{eqnarray*}
\lim_{t\rightarrow\infty}\sup\sup_{\Sigma}\log(-\Phi)-\log\frac{1}{n}\leq
\lim_{t\rightarrow\infty}\sup\widetilde{\psi}(t)+\lim_{t\rightarrow\infty}\sup\sup_{\Sigma}(\log\widetilde{\chi}-\widetilde{r}+\log2)\leq
0,
\end{eqnarray*}
which leads to
\begin{eqnarray*}
\lim_{t\rightarrow\infty}\inf\inf_{M}F\geq n.
\end{eqnarray*}
Then, together with \eqref{481}, we conclude that the following fact
\begin{equation*}
\sup_{\Sigma}|\kappa_{i}(t, \cdot)-1|\rightarrow 0, \ \
t\rightarrow\infty, \ \ \forall 1\leq i\leq n
\end{equation*}
is true.
\end{proof}

\begin{theorem}\label{109}
Under the assumptions of theorem \ref{main}, the principle
curvatures of the flow hypersurfaces of \eqref{3.01} converge to $1$
exponentially fast. There exists $C=C(n, \Sigma_0)$ such that for
all $(t, \xi)\in [0, \infty)\times \Sigma$, the estimate
\begin{equation*}
|h^{i}_{j}-\delta^{i}_{j}|\leq Ce^{-\frac{2t}{n}}
\end{equation*}
holds.
\end{theorem}

\begin{proof}
Define the function
\begin{equation*}
G(t, \xi)=\frac{1}{2}|h^{i}_{j}-\delta^{i}_{j}|^{2}(t, \xi), \forall
(t, \xi)\in [T, \infty)\times \Sigma.
\end{equation*}
Using the evolution equation \eqref{4.70} of $h^{i}_{j}$, we can get
the evolution equation of $G(t, \xi)$ as follows
\begin{eqnarray}\label{4.710}
\frac{\partial}{\partial t}G(t,
\xi)-\Phi^{\prime}F^{kl}\nabla_{k}\nabla_{l}G(t,
\xi)&=&(h^{i}_{j}-\delta_{j}^{i})\bigg(\Phi^{\prime}F^{kl}h_{kp}h^{p}_{l}h^{i}_{j}
-2\Phi^{\prime}Fh^{ip}h_{pj}+\Phi^{\prime\prime}F^i
F_j\nonumber\\&&+\Phi^{\prime}F^{kl,pq}h_{kl; j}h_{pq;}^{\ \ \
i}+\Phi^{\prime}F^{kl}g_{kl}h^{i}_{j}+O(e^{-\frac{(n+1)}{n}t})\bigg)\nonumber
\\ \nonumber&&-\Phi^{\prime}F^{kl}h_{j; k}^{i}h^{i}_{j; l}.
\end{eqnarray}
Set
$$G(t)=G(t, \xi_t)=\sup_{\xi\in\Sigma}G(t, \xi).$$
Since $$F^{kl}h_{j; k}^{i}h^{i}_{j; l}\geq C |\nabla A|^2$$ and
$$|h^{i}_{j}-\delta^{i}_{j}|\rightarrow 0,$$ so for large $t$ we can
absorb the terms involving the derivatives of $h^{i}_{j}$ by
$\Phi^{\prime}F^{kl}h_{j; k}^{i}h^{i}_{j; l}$. There holds the
following identity
\begin{equation*}
h^{i}_{k}h^{k}_{j}=(h^{i}_{k}-\delta^{i}_{k})(h^{k}_{j}-\delta^{k}_{j})+2(h^{i}_{j}-\delta^{i}_{j})+\delta^{i}_{j}.
\end{equation*}
Thus we have
\begin{eqnarray}\label{4.7100}
\frac{\partial}{\partial
t}G(t)&=&(h^{i}_{j}-\delta_{j}^{i})\bigg(\Phi^{\prime}F^{kl}h_{kp}h^{p}_{l}h^{i}_{j}
-2\Phi^{\prime}F(h^{i}_{p}-\delta^{i}_{p})(h^{p}_{j}-\delta^{p}_{j})\nonumber\\&&-4\Phi^{\prime}F(h^{i}_{j}-\delta^{i}_{j})
-2\Phi^{\prime}F+\Phi^{\prime}F^{kl}g_{kl}h^{i}_{j}+O(e^{-\frac{(n+1)}{n}t})\bigg)\\
\nonumber
&=&(h^{i}_{j}-\delta_{j}^{i})\bigg(\Phi^{\prime}F^{kl}(h_{kp}h^{p}_{l}-2h_{kl}+g_{kl})h^{i}_{j}
-2\Phi^{\prime}F(h^{i}_{p}-\delta^{i}_{p})(h^{p}_{j}-\delta^{p}_{j})\nonumber\\&&-2\Phi^{\prime}F(h^{i}_{j}
-\delta^{i}_{j})+O(e^{-\frac{(n+1)}{n}t})\bigg)\nonumber
\\ \nonumber&&
\end{eqnarray}
Choose Riemannian normal coordinates at $(t, \xi_t)$ such that at
this point we have
\begin{eqnarray*}
g_{ij}=\delta_{ij}, \ \ \ h_{ij}=\kappa_i \delta_{ij}, \ \ \
\kappa_1\leq \kappa_2\leq...\leq \kappa_{n}.
\end{eqnarray*}
For $t$ large enough, we can find
$\epsilon<\frac{4}{\sup_{\Sigma}F}$ such that
\begin{eqnarray}\label{909}
\frac{d}{dt}G(t)&\leq&\bigg(-\frac{4}{F}+2\Phi^{\prime}\sum_{j=1}^{n}|\kappa_j||\kappa_j-1|
\sum_{k=1}^{n}F^{kk}+\frac{4}{F^2}\max_{1\leq j\leq
n}|\kappa_j-1|\bigg)G(t)+\max_{1\leq j\leq
n}|\kappa_j-1|O(e^{-\frac{n+1}{n}t})\nonumber\\
\nonumber&\leq&(-\frac{4}{F}+\epsilon)G(t)+ \max_{1\leq j\leq
n}|\kappa_j-1|O(e^{-\frac{n+1}{n}t}).
\end{eqnarray}
Therefore, we have
\begin{eqnarray*}
G(t)\leq Ce^{-\mu_{1}t},
\end{eqnarray*}
where $\mu_1=\min\{\frac{4}{\sup_{M}F}-\epsilon, \frac{n+1}{n}\}>0$.
Thus,
\begin{eqnarray}\label{401}
|-\frac{4}{F}+\frac{4}{n}|\leq C\max_{i}|\kappa_i-1|\leq
Ce^{-\frac{1}{2}\mu_1 t}.
\end{eqnarray}
Now we define
\begin{equation*}
\overline{G}=\sup_{\Sigma}\frac{1}{2}|h^{i}_{j}-\delta^{i}_{j}|^{2}e^{\frac{4}{n}t}.
\end{equation*}
Similar to the process of getting \eqref{909}, we can obtain
\begin{eqnarray*}
\frac{d}{dt}\overline{G}&\leq&\bigg(-\frac{4}{F}+\frac{4}{n}+2\Phi^{\prime}\sum_{j=1}^{n}|\kappa_j||\kappa_j-1|
\sum_{k=1}^{n}F^{kk}+\frac{4}{F^2}\max_{1\leq j\leq
n}|\kappa_j-1|\bigg)\overline{G}+O(e^{-\frac{n+1}{n}t+\frac{4}{n}
t-\frac{1}{2}\mu_1 t})\\&\leq& Ce^{-\frac{1}{2}\mu_1
t}\overline{G}+O(e^{-\frac{n-3}{n}t-\frac{1}{2}\mu_1 t}),
\end{eqnarray*}
where we use \eqref{401} to get the last inequality. Thus,
\begin{equation*}
\overline{G}\leq C(n, \Sigma_0),
\end{equation*}
\end{proof}
which implies our result.

\begin{theorem}\label{800}
The estimate \eqref{4.61} in Lemma \ref{4.5} is true for
$\mu=\frac{2}{n}$.
\end{theorem}

\begin{proof}
Define
\begin{eqnarray*}
\widetilde{w}=\sup_{x\in \mathbb{S}^{n}}\frac{1}{2}|D\varphi(\cdot,
x)|^2e^{-\frac{2}{n}t}.
\end{eqnarray*}
The same calculation as in \eqref{201} implies
\begin{eqnarray*}
\frac{d}{dt}\widetilde{w} \leq -\frac{2}{F^2}
\frac{\lambda^{\prime\prime}}{\lambda}F^{kl}\widetilde{g}_{kl}\widetilde{w}+\frac{2}{n}\widetilde{w}\leq
-\frac{2n}{F^2}
\frac{\lambda^{\prime\prime}}{\lambda}\widetilde{w}+\frac{2}{n}\widetilde{w}.
\end{eqnarray*}
\end{proof}
Recalling the estimate \eqref{00.2}
\begin{eqnarray*}
\frac{\lambda^{\prime\prime}}{\lambda}=1-\frac{1}{2}m(1-n)\lambda^{-n-2}=1+O(e^{-\frac{n+2}{n}t}).
\end{eqnarray*}
Together with \eqref{401}, we have
\begin{eqnarray*}
\frac{d}{dt}\widetilde{w} \leq Ce^{-\frac{1}{2}\mu_1
t}\widetilde{w},
\end{eqnarray*}
which implies $\widetilde{w}$ is bounded from above. Therefore, the
theorem holds.

\begin{theorem}\label{010}
Under the assumptions of Theorem \ref{main}. There exists a constant
$C=C(n, \Sigma_0)$ such that
\begin{equation*}
\Big|D^2 \varphi\Big|\leq Ce^{-\frac{t}{n}}.
\end{equation*}
\end{theorem}

\begin{proof}
Recalling \eqref{2.07}, we have
\begin{equation*}
\varphi^{i}_{j}=v^{-2}\varphi^{i}\varphi^{k}\varphi_{kj}+\lambda^{\prime}\delta^{i}_{j}-v\lambda
h^{i}_{j}.
\end{equation*}
From Lemma \ref{4.01}, we get
\begin{eqnarray*}
|\lambda^{\prime}-\lambda|=\frac{1-m\lambda^{1-n}}{\lambda\sqrt{1+\lambda^2-m\lambda^{1-n}}}\leq\frac{1}{\lambda}\leq
Ce^{-\frac{1}{n}t}.
\end{eqnarray*}
Together with Theorems \ref{010} and \ref{109}, we obtain
\begin{eqnarray*}
|D^{2}\varphi|&\leq& C|D \varphi|^2
|D^{2}\varphi|+|\lambda^{\prime}\delta^{i}_{j}-\lambda
\delta^{i}_{j}|+|\lambda \delta^{i}_{j}-v\lambda
\delta^{i}_{j}|+|v\lambda \delta^{i}_{j}-v\lambda h^{i}_{j}|\\&
\leq& Ce^{-\frac{2}{n}t}|D^{2}\varphi|+Ce^{-\frac{1}{n}t}.
\end{eqnarray*}
Choosing $T$ large enough ($Ce^{-\frac{2}{n}t}<\frac{1}{2}$), we
know that the estimate
\begin{equation*}
|D^{2}\varphi|\leq Ce^{-\frac{1}{n}t}
\end{equation*}
holds  for $t>T$.
\end{proof}

Clearly, from Theorem \ref{010}, we can show that there exists a
constant $C=C(n, \Sigma_0)$ such that
\begin{equation*}
\parallel D^2 r\parallel_{\mathbb{S}^{n}}\leq C.
\end{equation*}
Then by Krylov-Safonov estimate \cite{Kr}, we have
\begin{equation*}
\parallel r\parallel_{C^{2, \alpha}({\mathbb{S}^{n}})}\leq C(n,
\Sigma_0),
\end{equation*}
which implies the following conclusion.
\begin{theorem}\label{0100}
Under the assumptions of theorem \ref{main}. The function
$$\widetilde{r}(t, \theta)=r(t, \theta)-\frac{t}{n}$$
 converge to a well-defined $C^2$
function $f(\theta)$ in $C^{2, \alpha}$.
\end{theorem}

\begin{proof}
Because of the boundedness of $\widetilde{r}=r-\frac{t}{n}$ in
$C^{2}(\mathbb{S}^n)$, we only have to show the pointwise limit
$$\lim_{t\rightarrow\infty}(r-\frac{t}{n})$$
exists for all $x\in \mathbb{S}^n$. We have
$$\frac{\partial}{\partial t}\widetilde{r}=\frac{v}{F}-\frac{1}{n}=\frac{v-1}{F}+\frac{n-F}{nF}\geq-C(n, \Sigma_0)e^{-\frac{t}{n}}.$$
Thus,
$$(\widetilde{r}-nCe^{-\frac{t}{n}})^{\prime}\geq 0,$$
which implies the result.
\end{proof}

\begin{remark}
\rm{Following the techniques in \cite[Section 6]{Ge2} and
\cite[Section 5]{Sc1}, we may also get estimates of high order for
$\widetilde{r}$
\begin{equation*}
\parallel \widetilde{r}\parallel_{C^{k}({\mathbb{S}^n})}\leq C(n,
\Sigma_0), \ \ \forall k\in \mathbb{N}.
\end{equation*}
Therefore, the $C^{\infty}$ convergence in the above theorem may be
obtained.}
\end{remark}


\vspace{1cm}



\begin{thebibliography}{50}
\setlength{\itemsep}{-0pt} \small

\bibitem{Di} Q. Ding, The inverse mean curvature ow in rotationally symmetric
spaces, Chin. Ann. Math., Ser. B 32 (2011), No. 1, 27-44.

\bibitem{Be} Brendle, Simon; Hung, Pei-Ken; Wang, Mu-Tao, A Minkowski inequality
for hypersurfaces in the anti-de Sitter-Schwarzschild manifold,
Comm. Pure Appl. Math. 69 (2016), no. 1, 124-144.

\bibitem{Ec} K. Ecker; G. Huisken, mmersed hypersurfaces with constant
Weingarten curvature, Math. Ann. 283 (1989), No. 2, 329-332.

\bibitem{Gw1} Ge, Yuxin; Wang, Guofang; Wu, Jie; Xia, Chao, A Penrose inequality
for graphs over Kottler space, Calc. Var. Partial Differential
Equations 52 (2015), no. 3-4, 755-782.

\bibitem{Gw2} Ge, Yuxin; Wang, Guofang; Wu, Jie, The GBC mass for asymptotically
hyperbolic manifolds, Math. Z. 281 (2015), no. 1-2, 257-297.

\bibitem{Gw3} Ge, Yuxin; Wang, Guofang; Wu, Jie, Hyperbolic Alexandrov-Fenchel
quermassintegral inequalities II, J. Differential Geom. 98 (2014),
no. 2, 237-260.

\bibitem{Ge1} Claus Gerhardt, Flow of nonconvex hypersurfaces into spheres, J.
Differ. Geom. 32 (1990), 299-314.

\bibitem{Ge2} Claus Gerhardt, Inverse curvature flows in hyperbolic space, J. Differ.
Geom. 89 (2011), 487-527.

\bibitem{Ge3} Claus Gerhardt, Curvature Problems, Ser. in Geom. and Topol., vol.
39, International Press, Somerville, MA, (2006).

\bibitem{Ge4} C. Gerhardt, Curvature estimates for Weingarten hypersurfaces in
Riemannian manifolds, Adv. Calc. Var. 1 (2008), 123-132.

\bibitem{Ge5} Claus Gerhardt, Curvature flows in the sphere, J. Differential Geom.
100 (2015), no. 2, 301-347.

\bibitem{Ge6} Claus Gerhardt,  Non-scale-invariant inverse curvature flows in
Euclidean space, Calc. Var. Partial Differential Equations 49
(2014), no. 1-2, 471-489.

\bibitem{Ge7} Claus Gerhardt, Closed Weingarten hypersurfaces in space forms,
Geom. Anal. and the Calc. of Var. (Jurgen Jost, ed.), International
Press, Boston, (1996).

\bibitem{Hung} Pei-Ken Hung and Mu-Tao Wang, Inverse mean curvature flows in the
hyperbolic 3-space revisited, Calc. Var. Partial Differential
Equations 54 (2015), no. 1, 119-126.

\bibitem{Hu} Gerhard Huisken, Flow by mean curvature of convex surfaces into
spheres., J. Differ. Geom. 20 (1984), 237-266.

\bibitem{Li} Haizhong Li; Yong Wei, On inverse mean curvature flow in Schwarzschild space and Kottler
space, available online at arXiv:1212.4218.

\bibitem{Lu} Siyuan Lu, Inverse curvature flow in  anti-de sitter-schwarzschild
manifold, available online at arXiv:1609.09733v1

\bibitem{Ma} M. Makowski; Julian Scheuer, Rigidity results, inverse curvature flows and
Alexandrov-Fenchel-type inequalities in the sphere, (2013), to
appear in Asian J. Math., and available online at arxiv:1307.5764.

\bibitem{Ne} A. Neves, Insufficient convergence of inverse mean curvature flow on
asymptotically hyperbolic mani- folds, J. Differential Geom. 84
(2010), no. 1, 191-229.

\bibitem{Kr} N.V. Krylov, Nonlinear elliptic and parabolic equations of the
second order, Reidel, Dordrecht, (1987).

\bibitem{Pe} Li. P, Harmonic functions and applications to complete
manifolds, University of California, Irvine, 2004, preprint.

\bibitem{Sc1} Julian Scheuer, Non-scale-invariant inverse curvature flows in hyperbolic
space, Calc. Var. Partial Differential Equations 53 (2015), no. 1-2,
91-123.

\bibitem{Sc2} Julian Scheuer, The inverse mean curvature flow in warped cylinders of non-positive
radial curvature, available online at arXiv:1312.5662.

\bibitem{Ur} Urbas J., On the expansion of starshaped hypersurfaces by symmetric functions of their principal
curvatures, Math. Z. 205 (1990), no. 3, 355-372.

\bibitem{Ur1} Urbas J., An expansion of convex hypersurfaces, J.
Differential Geom. 33 (1991), no. 1, 91-125.

\end{thebibliography}
\end{document}